\documentclass{amsart}
\usepackage{amsfonts}
\usepackage{amsmath,amssymb}
\usepackage{amsthm}
\usepackage{amscd}
\usepackage{graphics}
\usepackage{graphicx}

\theoremstyle{remark}{

\newtheorem{Ex}{{\rm Example}}

\newtheorem{Prob}{{\rm Problem}}

}
\theoremstyle{plain}
{

\newtheorem{Thm}{Theorem}

}

\begin{document}
\title[Morse functions with given level sets on $3$-dimensional manifolds]{On reconstructing Morse functions with prescribed level sets on $3$-dimensional manifolds and a necessary and sufficient condition for the reconstruction}
\author{Naoki kitazawa}

\address{Osaka Central Advanced Mathematical Institute (OCAMI) \\
	3-3-138 Sugimoto, Sumiyoshi-ku Osaka 558-8585
	TEL: +81-6-6605-3103
}
\email{naokikitazawa.formath@gmail.com}
\urladdr{https://naokikitazawa.github.io/NaokiKitazawa.html}
\maketitle
\begin{abstract}
We discuss a necessary and sufficient condition for reconstruction of Morse functions with prescribed ({\it regular}) {\it level sets} on $3$-dimensional manifolds. The present work strengthens a previous result of the author where only sufficient conditions are studied. Our new work is also regarded as a kind of addenda.
\end{abstract}
%【REVISE】 combinatoric ～ is → combinatorial object. It is .
%【REVISE】  such that a point is a vertex if and only if the corresponding connected component of the level set contains some singular points → whose vertex set is the set of all points containing some singular points in the corresponding connected component of the level set .
%【REVISE】 We delete "extending the result before".
\section{Introduction.}
\label{sec:1}

{\it Morse} functions have been fundamental and important tools in mathematics, mainly in various geometry of manifolds, since the birth of the theory of Morse functions in the 20th century. 

For arguments, we introduce fundamental terminologies, notions and notation. Let ${\mathbb{R}}^n$ denote the $n$-dimensional Euclidean space, which is one of simplest smooth manifolds. Let $\mathbb{R}:={\mathbb{R}}^1$.
For a smooth map $c:X \rightarrow Y$ between smooth manifolds, a {\it singular} point $p \in X$ {\it of} $c$ means a point where the rank of the differential of $c$ at $p$ is smaller than both the dimensions of $X$ and $Y$. A {\it singular value} $c(p)$ {\it of} $c$ is a value realized as a value of $c$ at a singular point $p$ of $c$. 
We also use "{\it critical}" instead of "{\it singular}" in the real-valued function case $Y:=\mathbb{R}$.
For a real valued function $c$, each preimage $c^{-1}(q)$ is called a {\it level set}. If this contains no critical point, it is also called a {\it regular} level set.
A {\it Morse} function $c:X \rightarrow \mathbb{R}$ is a smooth function whose critical point is always in the interior of the $m$-dimensional manifold $X$ and whose critical point $p$ is of the form $c(x_1 ,\cdots, x_m)={\Sigma}_{j=1}^{m-i(p)} {x_j}^2-{\Sigma}_{j=1}^{i(p)} {x_{m-i(p)+j}}^2+c(p)$ where a suitable local coordinate and an integer $0 \leq i(p) \leq m$ are chosen. The integer $i(p)$ is shown to be chosen uniquely and it is the {\it index} of the critical point $p$ of $c$ and critical points of $c$ appear discretely. The textbooks \cite{milnor1, milnor2} are on related fundamental theory and the textbook \cite{golubitskyguillemin} is on related singularity theory.

Surprisingly, very fundamental problems on Morse functions are still studied and developing. Our study is on reconstruction of Morse functions or more generally, nice smooth functions with prescribed (regular) level sets. This is very new, and founded in \cite{sharko}. The papers \cite{masumotosaeki, michalak} are studies following \cite{sharko}. These studies are on smooth functions which are at critical points represented by certain elementary polynomials or Morse functions, on closed surfaces. 
The author has studied such reconstruction respecting manifolds of (regular) level sets in \cite[Theorem 1]{kitazawa1}. 

The {\it Reeb space} $R_c$ of a smooth real-valued function $c:X \rightarrow \mathbb{R}$ on an $m$-dimensional closed manifold $X$ is the space of all connected components of all level sets and defined as a quotient space of $X$. In some nice cases, such spaces are digraphs ({\it Reeb digraphs}) naturally. The vertex set of this consists of all connected components containing some critical points. The orientation of the digraph is induced from the values of the function, canonically.
The author has constructed Morse functions, more generally, so-called {\it Morse-Bott} functions, or functions of suitably generalized classes, on suitable closed, connected and orientable manifolds of dimension $m=3$ whose {\it Reeb digraphs} are isomorphic to given non-trivial finite and acyclic digraphs and whose regular level sets are prescribed closed and connected surfaces. In \cite{saeki2}, Saeki considers a very general situation and methods, where types of critical points or singularities of the functions are not explicitly presented. In \cite[Main Theorems 1 and 2]{kitazawa3}, the author considers a case generalizing \cite[Theorem 1]{kitazawa1}. Level sets containing no critical point are generalized to ($m-1$)-dimensional closed and connected manifolds of boundaries of $m$-dimensional compact and connected manifolds of a certain class (where $m \geq 2$): {\it most fundamental handlebodies} ({\it MFHs}). The class of boundaries of $m$-dimensional MFHs contains an ($m-1$)-dimensional (standard) sphere $S^{m-1}$, manifolds of the form ${\sharp}_{j=1}^l (S^{k_j} \times S^{m-k_j-1})$ with $l \geq 0$ and $1 \leq k_j \leq m-2$, and all $3$-dimensional closed, connected and orientable manifolds (in the case $m=4$). A manifold diffeomorphic to the boundary of an MFH is also called a {\it BMFH} in \cite{kitazawa3}. In \cite[Main Theorems 1 and 2]{kitazawa2}, the author has extended \cite[Theorem 1]{kitazawa1} to the non-orientable case and presented cases we cannot discuss by \cite{kitazawa3}. For example, in the case the dimension of the manifold $X$ is $m=3$, a projective plane ${\mathbb{R}P}^2$ can appear as a connected component of a regular level set of the function and this cannot be the boundary of any $3$-dimensional compact manifold. 

We formulate our problem.  
\begin{Prob}
\label{prob:1}
Let $m \geq 2$ be an integer, Let $a<s<b$ be three real numbers.
Let $F_{a}$ and $F_{b}$ be {\rm (}$m-1${\rm )}-dimensional closed manifolds which may not be connected.
Can we reconstruct some Morse function ${\tilde{f}}_{F_{a},F_{b}}:{\tilde{M}}_{F_{a},F_{b}} \rightarrow \mathbb{R}$ on a suitable $m$-dimensional compact and connected manifold ${\tilde{M}}_{F_{a},F_{b}}$ enjoying the following properties.
\begin{enumerate}
\item The image is $[a,b]:=\{t \mid a \leq t \leq b\}$. The number of critical values of ${\tilde{f}}_{F_{a},F_{b}}:{\tilde{M}}_{F_{a},F_{b}} \rightarrow \mathbb{R}$ is $1$ and the unique value is $s$.
\item The regular level set ${{\tilde{f}}_{F_{a},F_{b}}}^{-1}(a)$ (${{\tilde{f}}_{F_{a},F_{b}}}^{-1}(b)$) is diffeomorphic to $F_{a}$ (resp. $F_{b}$). The level set ${{\tilde{f}}_{F_{a},F_{b}}}^{-1}(s)$ is connected. The boundary of the manifold ${\tilde{M}}_{F_{a},F_{b}}$ is ${{\tilde{f}}_{F_{a},F_{b}}}^{-1}(a) \sqcup {{\tilde{f}}_{F_{a},F_{b}}}^{-1}(b)$.
\end{enumerate}
  
\end{Prob}
Michalak has solved the case $m=2$ completely and affirmatively in \cite{michalak}. In \cite{kitazawa1}, followed by \cite{kitazawa3}, the author has solved the case $m=3,4$ with $F_a$ and $F_b$ being orientable. In \cite{kitazawa2}, the author has partially solved the case $m=3$ where these surfaces may not be orientable. Our new result is a necessary and sufficient condition for Problem \ref{prob:1} to be solved affirmatively. In \cite{kitazawa2}, the author has found a sufficient condition only. In these studies, explicit construction of explicit local Morse functions is essential. In addition, in considering non-orientable surfaces as regular level sets in \cite{kitazawa2}, some elementary and original combinatorial arguments are also essential and they are not in \cite{kitazawa1} or \cite{kitazawa3}.

We review and introduce several properties on closed surfaces. We also introduce elementary and numerical topological invariants $P(S)$ and $P_{\rm o}(S)$ for closed surfaces $S$ which may not be connected.
 \begin{itemize}
 \item A closed and connected surface $S$ is diffeomorphic to a connected sum $({\sharp}_{j_1=1}^{l_1} (S^1 \times S^1)) {\sharp} ({\sharp}_{j_2=1}^{l_2} {{\mathbb{R}}P}^2)$.
 \item In the previous scene, $S$ is orientable if and only if $l_2=0$ and $S$ is a sphere if and only if $l_1=l_2=0$. If $S$ is non-orientable, then we can choose $l_1=0$. 
 \item In the previous representations, the integer $l_2=0$ gives a topological invariant for closed, connected and orientable surfaces $S$. Let $P(S):=l_2=0$ for a closed, connected, and orientable surface $S$ and in this situation, $l_1$ gives a topological invariant for closed, connected and orientable surfaces $S${\rm :} this is the so-called ({\it orientable}) {\it genus} of a closed, connected and orientable surface $S$. The integer $l_2$ also gives a topological invariant for closed, connected, and non-orientable surfaces $S$ under the constraint that $l_1=0$ and in this scene, let $P(S):=l_2${\rm :} this is the so-called {\it non-orientable genus} of a closed, connected and non-orientable surface $S$.
\item We can extend $P(S)$ for closed surfaces $S$ which may not be connected, in the additive way, canonically. In other words, for the disjoint union $S={\sqcup}_{j=1}^l S_{{\rm c},j}$ of $l_0>0$ closed and connected surfaces $S_{{\rm c},j}$, $P(S)={\Sigma}_{j=1}^l P(S_{{\rm c},j})$ . The integer $P(S)$ gives a numerical topological invariant for closed surfaces $S$. 
\item Let $P_{\rm o}(S) \leq P(S)$ denote the number of connected components $S_{{\rm c},j}$ of the closed surface $S={\sqcup}_{j=1}^l S_{{\rm c},j}$ with $P(S_{{\rm c},j})$ being odd.
 \end{itemize}

Theorem \ref{thm:1} is our main result.  
\begin{Thm}[An extension of \cite{kitazawa2}]
\label{thm:1}
Problem \ref{prob:1} is solved affirmatively in the case $m=3$ if and only if the following hold.
\begin{enumerate}
\item \label{thm:1.1} The value $P_{\rm o}(F_b)-P_{\rm o}(F_a)$ is even. This is equivalent to the fact that the closed surface $F_a \sqcup F_b$ is the boundary of some $3$-dimensional compact and connected manifold{\rm :} in short two closed surfaces $F_a$ and $F_b$ are {\it cobordant}.
\item \label{thm:1.2} Both the relations $P_{\rm o}(F_b) \leq P(F_a)$ and $P_{\rm o}(F_a) \leq P(F_b)$ hold.
\end{enumerate}
\end{Thm}

We prove this in the next section. We first review \cite{kitazawa2} to check that our conditions (\ref{thm:1.1}) and (\ref{thm:1.2}) give a sufficient condition for Problem \ref{prob:1} (Theorem \ref{thm:3} with Theorem \ref{thm:2}). 
It is a main ingredient to apply well-known correspondence between critical points of Morse functions and so-called {\it handles} and arguments similar to \cite[Lemma 6.6]{saeki1} and related ones to show that these conditions (\ref{thm:1.1}) and (\ref{thm:1.2}) also give a necessary condition. %Such arguments are applied by the author in \cite{kitazawa4} and our paper also presents another new application.

%In \cite{kitazawa4} the author classified Morse functions on $3$-dimensional closed, connected and orietable manifolds such that preimages of single points containing no singular points of the functions are disjoint unions of copies of the sphere $S^2$ and the torus $S^1 \times S^1$, via graphs or so-called {\it Reeb graphs} and data on preimages. 
%The presented study has strengthened \cite[Theorem 6.5]{saeki1}, characterizing $3$-dimensional closed, connected and orientable manifolds represented as connected sums of copies of the product $S^1 \times S^2$ and so-called {\it Lens spaces} by the existence of such Morse functions.

 Last, the third section is for our conclusion and additional remark.

\section{On our main result, Theorem \ref{thm:1}.}
\subsection{The Reeb graph of a smooth function.}
We explain the {\it Reeb space} ({\it graph} and {\it digraph}) of a smooth function $c:X \rightarrow \mathbb{R}$ again and this does not contradict the short exposition in the previous section.

The {\it Reeb space} of a smooth function is the space of all connected components of level sets (\cite{reeb}). Reeb spaces are regarded as graphs naturally in considerable cases (\cite[Theorem 3.1]{saeki2}: see also \cite{izar} in the Morse function case).
We consider a Morse function $c:X \rightarrow \mathbb{R}$ on a closed manifold $X$ or a Morse function ${\tilde{f}}_{F_{a},F_{b}}:{\tilde{M}}_{F_{a},F_{b}} \rightarrow \mathbb{R}$ in Problem \ref{prob:1}. We also use $c:={\tilde{f}}_{F_{a},F_{b}}:X:={\tilde{M}}_{F_{a},F_{b}} \rightarrow \mathbb{R}$ in the latter case. 
We can define the equivalence relation ${\sim}_c$ on $X$ by the following: $x_1 {\sim}_c x_2$ if and only if $x_1$ and $x_2$ are in a same connected component of some level set $c^{-1}(y)$ ($y \in \mathbb{R}$). The {\it Reeb space $R_c:=X / {\sim}_c$ of $c$} is defined as the quotient space, has the structure of a graph by the following rule and is also called the {\it Reeb graph} of $c$: a point $v \in R_c$ is a vertex if and only if $v$ satisfies either the following where $q_c:X \rightarrow R_c$ denotes the quotient map.
\begin{itemize}
\item The connected component ${q_c}^{-1}(v)$ contains some critical points of $c$. 
\item The connected component ${q_c}^{-1}(v)$ is a connected component of ${{\tilde{f}}_{F_{a},F_{b}}}^{-1}(a) \sqcup {{\tilde{f}}_{F_{a},F_{b}}}^{-1}(b)$ where $c:={\tilde{f}}_{F_{a},F_{b}}$.
\end{itemize}

Respecting the values of the function $c$, edges of the Reeb graph are oriented. 
In other words, an edge $e$ incident to exactly two vertices $v_{e,1}$ and $v_{e,2}$ is oriented from $v_{e,1}$ to $v_{e,2}$ if the value of $c$ on ${q_c}^{-1}(v_{e,2})$ is greater than that on ${q_c}^{-1}(v_{e,1})$. Thus, the Reeb graph is regarded as a digraph (the {\it Reeb digraph} $\overrightarrow{R_c}$ of $c$).  
 
%However, related knowledge may help us to understand our discussions. Related to this, we do not assume arguments on our preprints \cite{kitazawa3}, for example.
\subsection{Our previous result related to sufficiency for Theorem \ref{thm:1} and checking the sufficiency.}
Before our proof, we review \cite[Main Theorem 2]{kitazawa2}.

This is regarded as a theorem on a sufficient condition for Problem \ref{prob:1} to be solved affirmatively. 

For the graph $G$ in \cite[Main Theorem 2]{kitazawa2}, $g:G \rightarrow \mathbb{R}$ is a continuous function which is injective on each edge. This function makes $G$ a digraph $\overrightarrow{G}$ each edge of which is oriented as in the case of Reeb digraphs. An integer-valued function $r_G$ on the edge set $E_G$ of the graph $G$ is also associated.
Consider a vertex $v \in G$ which at least one edge enters and which at least one edge departs from. 

Second, the set $A_{{\rm up},v}$ ($A_{{\rm low},v}$) in \cite{kitazawa2} is the set of all edges $e$ departing from (resp. entering) $v$ and with $r_G(e)$ being odd and negative. The number ${\deg}_{{\rm out},r_G<0,r_G \text{ is odd.}}(v)$ 
(resp. ${\deg}_{{\rm in},r_G<0,r_G \text{ is odd.}}(v)$) stands for the size of the set. 

Third, the set $B_{{\rm up},v}$ ($B_{{\rm low},v}$) is the set of all edges $e$ satisfying the following: the integer $r_G(e)$ is negative and the edge $e$ departs from (resp. enters) the vertex $v$. 

The integer ${r^{\prime}}_G(e)$ is defined for an edge $e \in B_{{\rm up},v} \bigcup B_{{\rm low},v}$ and as the maximal even integer smaller than or equal to the absolute value $|r_G(e)|$ of $r_G(e)$. If $e \notin  B_{{\rm up},v} \bigcup B_{{\rm low},v}$, then as an original rule here, let ${r^{\prime}}_G(e):=0$.

Based on this, we present \cite[Main Theorem 2]{kitazawa2}.
\begin{Thm}
\label{thm:2}
In the situation above, let $G_v$ be a small connected neighborhood of $v$ in $G$. We see this as a digraph $\overrightarrow{G_v}$ induced canonically from $\overrightarrow{G}${\rm :} let $\{e_{a,j}\}_{j=1}^{i_a}$ {\rm (}resp. $\{e_{b,j}\}_{j=1}^{i_b}${\rm )} denote the edges entering {\rm (}resp. departing from{\rm )} $v$. 
We define the values $r_G(e_{a,j})$ and $r_G(e_{b,j})$ and the values ${r^{\prime}}_G(e_{a,j})$ and ${r^{\prime}}_G(e_{b,j})$ canonically{\rm :} we consider the unique edge containing each egde in $G$ and we define the values as the values of the original functions $r_G$ and ${r^{\prime}}_G$ at these edges of $G$.

Suppose also the following.
\begin{enumerate}
\item \label{thm:2.1} Either of the following holds.
\begin{enumerate}
\item  \label{thm:2.1.1} ${\deg}_{{\rm out},r_G<0,r_G \text{ is odd.}}(v)={\deg}_{{\rm in},r_G<0,r_G \text{ is odd.}}(v)$. 
\item  \label{thm:2.1.2} ${\deg}_{{\rm out},r_G<0,r_G \text{ is odd.}}(v)-{\deg}_{{\rm in},r_G<0,r_G \text{ is odd.}}(v)$ is positive, even, and smaller than or equal to ${\Sigma}_{e \in B_{{\rm low},v}} {r^{\prime}}_G(e)$.
\item  \label{thm:2.1.3} ${\deg}_{{\rm in},r_G<0,r_G \text{ is odd.}}(v)-{\deg}_{{\rm out},r_G<0,r_G \text{ is odd.}}(v)$ is positive, even, and smaller than or equal to ${\Sigma}_{e \in B_{{\rm up},v}} {r^{\prime}}_G(e)$.
\end{enumerate}
\item \label{thm:2.2} Let $F_a$ be a closed surface $F_a={\sqcup}_{j=1}^{i_a} F_{a,j}$ with $F_{a,j}$ being connected and $F_b$ a closed surface $F_b={\sqcup}_{j=1}^{i_b} F_{b,j}$ with $F_{b,j}$ being connected. Furthermore, the following hold.
\begin{enumerate}
\item \label{thm:2.2.1} The surface $F_{a,j}$ is diffeomorphic to ${\sharp}_{j=1}^{r_G(e_{a,j})} (S^1 \times S^1)$ if $r_G(e_{a,j}) \geq 0$ and $F_{a,j}$ is diffeomorphic to ${\sharp}_{j=1}^{-r_G(e_{a,j})} {\mathbb{R}P}^2$ if $r_G(e_{a,j})<0$.
\item \label{thm:2.2.2} The surface $F_{b,j}$ is diffeomorphic to ${\sharp}_{j=1}^{r_G(e_{b,j})} (S^1 \times S^1)$ if $r_G(e_{b,j}) \geq 0$ and $F_{b,j}$ is diffeomorphic to ${\sharp}_{j=1}^{-r_G(e_{b,j})} {\mathbb{R}P}^2$ if $r_G(e_{b,j})<0$.
\end{enumerate}
\end{enumerate}
In this situation, we can construct a Morse function ${\tilde{f}}_{F_{a},F_{b}}:{\tilde{M}}_{F_{a},F_{b}} \rightarrow \mathbb{R}$ as in Problem \ref{prob:1} whose Reeb digraph is isomorphic to $\overrightarrow{G_v}$ through an isomorphism
 $\phi:\overrightarrow{R_{{\tilde{f}}_{F_{a},F_{b}}}} \rightarrow \overrightarrow{G_v}$ mapping the vertex ${v_{a,j}}$ with ${q_{{\tilde{f}}_{F_{a},F_{b}}}}^{-1}(v_{a,j})=F_{a,j}$ to the vertex from which $e_{a,j}$ departs and the vertex ${v_{b,j}}$ with ${q_{{\tilde{f}}_{F_{a},F_{b}}}}^{-1}(v_{b,j})=F_{b,j}$ to the vertex which $e_{b,j}$ enters.
\end{Thm}
We have the following.
\begin{Thm}\label{thm:3}
In Theorem \ref{thm:1}, the two conditions {\rm (}\ref{thm:1.1}{\rm )} and {\rm (}\ref{thm:1.2}{\rm )} give a sufficient condition to solve Problem \ref{prob:1} affirmatively in the case $m=3$.
\end{Thm}
\begin{proof}
To complete the proof, it is sufficient to prove that we can have a situation of Theorem \ref{thm:2}, respecting the conditions of Theorem \ref{thm:1}. We prove this.

Let $F_a={\sqcup}_{j=1}^{k_a} F_{a,j}$ and $F_b={\sqcup}_{j=1}^{k_b} F_{b,j}$ with $k_a$ and $k_b$ being positive integers and $F_{a,j}$ and $F_{b,j}$ being connected. \\
\ \\
Case 1 $P_{\rm o}(F_a)=P_{\rm o}(F_b)$.\\
In this case, in the assumption of Theorem \ref{thm:2}, we put $i_a:=k_a$ and $i_b:=k_b$ and we define $r_G(e_{a,j})$ and $r_G(e_{b,j})$ as follows.
\begin{itemize}
\item $r_G(e_{a,j}):=\text{(The orientable genus of $F_{a,j}$)}$ if $F_{a,j}$ is orientable and $r_G(e_{a,j}):=-\text{(The non-orientable genus of $F_{a,j}$)}=P(F_{a,j})$ if $F_{a,j}$ is non-orientable.
\item $r_G(e_{b,j}):=\text{(The orientable genus of $F_{b,j}$)}$ if $F_{b,j}$ is orientable and $r_G(e_{b,j}):=-\text{(The non-orientable genus of $F_{b,j}$)}=P(F_{b,j})$ if $F_{b,j}$ is non-orientable.
\end{itemize}
This guarantees the condition (\ref{thm:2.2}) of Theorem \ref{thm:2}.
By our definitions, we also have the following fundamental facts.
\begin{itemize}
	\item $P(F_{a,j})=0$, $e_{a,j} \notin B_{{\rm low},v}$, and
	${r^{\prime}}_G(e_{a,j})=0$ if $F_{a,j}$ is orientable. $e_{a,j} \in B_{{\rm low},v}$ and ${r^{\prime}}_G(e_{a,j})$ is the maximal even number smaller than or equal to $-r_G(e_{a,j})=P(F_{a,j})> 0$ if $F_{a,j}$ is non-orientable.
	\item $P(F_{b,j})=0$, $e_{b,j} \notin B_{{\rm up},v}$, and
	${r^{\prime}}_G(e_{b,j})=0$ if $F_{b,j}$ is orientable. $e_{b,j} \in B_{{\rm up},v}$ and ${r^{\prime}}_G(e_{b,j})$ is the maximal even number smaller than or equal to $-r_G(e_{b,j})=P(F_{b,j})>0$ if $F_{b,j}$ is non-orientable.
	\item $P_{\rm o}(F_{a,j})=-r_G(e_{a,j})-{r^{\prime}}_G(e_{a,j})=1$ ($P_{\rm o}(F_{a,j})=0$) if $e_{a,j} \in A_{{\rm low},v}$ (resp. $e_{a,j} \notin A_{{\rm low},v}$). $P_{\rm o}(F_{b,j})=-r_G(e_{b,j})-{r^{\prime}}_G(e_{b,j})=1$ ($P_{\rm o}(F_{b,j})=0$) if $e_{b,j} \in A_{{\rm up},v}$ (resp. $e_{b,j} \notin A_{{\rm up},v}$).
	\item $-r_G(e_{a,j})={r^{\prime}}_G(e_{a,j})$ (${r^{\prime}}_G(e_{a,j})=0$) if $e_{a,j} \in B_{{\rm low},v}-A_{{\rm low},v}$ (resp. $e_{a,j} \notin B_{{\rm low},v}$). $-r_G(e_{b,j})={r^{\prime}}_G(e_{b,j})$ (${r^{\prime}}_G(e_{b,j})=0$) if $e_{b,j} \in B_{{\rm up},v}-A_{{\rm up},v}$ (resp. $e_{b,j} \notin B_{{\rm up},v}$). 
	\end{itemize}
We also have the condition (\ref{thm:2.1.1}) of Theorem \ref{thm:2} by the relation $P_{\rm o}(F_a)=P_{\rm o}(F_b)$ with the relations $P_{\rm o}(F_a)={\deg}_{{\rm in},r_G<0,r_G \text{ is odd.}}(v)$ and $P_{\rm o}(F_b)={\deg}_{{\rm out},r_G<0,r_G \text{ is odd.}}(v)$.
\\
\ \\
Case 2 The difference $P_{\rm o}(F_b)-P_{\rm o}(F_a)$ is positive and even.\\
In this case, in the situation of Theorem \ref{thm:2}, we put $(i_a,i_b,r_G(e_{a,j}),r_G(e_{b,j}))$ as before. This guarantees the condition (\ref{thm:2.2}) of Theorem \ref{thm:2}.

We remember the condition (\ref{thm:1.2}) of Theorem \ref{thm:1}. In other words, $P_{\rm o}(F_b) \leq P(F_a)$, where the relation $P_{\rm o}(F_a) \leq P(F_b)$ holds by the condition on this case $P_{\rm o}(F_a)<P_{\rm o}(F_b)$ with $P_{\rm o}(F_b) \leq P(F_b)$. By adding $-P_{\rm o}(F_a)$ to the both sides of $P_{\rm o}(F_b) \leq P(F_a)$, we have $0<P_{\rm o}(F_b)-P_{\rm o}(F_a) \leq P(F_a)-P_{\rm o}(F_a)$. By our definitions and fundamental facts as presented in Case 1, we have
\begin{eqnarray*}
	0<P_{\rm o}(F_b)-P_{\rm o}(F_a) \leq P(F_a)-P_{\rm o}(F_a) &=&{\Sigma}_{j=1}^{i_a} (P(F_{a,j}))-P_{\rm o}(F_a)\\
	&=&{\Sigma}_{e_{a,j} \in {B_{{\rm low},v}}} (-r_G(e_{a,j}))-P_{\rm o}(F_a)\\
	&=&{\Sigma}_{e_{a,j} \in {B_{{\rm low},v}}} (-r_G(e_{a,j}))-{\Sigma}_{e_{a,j} \in {A_{{\rm low},v}}} 1\\
	&=&{\Sigma}_{e_{a,j} \in {B_{{\rm low},v}}} ({r^{\prime}}_G(e_{a,j}))
\end{eqnarray*}
and we have the inequality $P_{\rm o}(F_b)-P_{\rm o}(F_a) \leq {\Sigma}_{j=1}^{i_a} ({r^{\prime}}_G(e_{a,j}))={\Sigma}_{e_{a,j} \in {B_{{\rm low},v}}} ({r^{\prime}}_G(e_{a,j}))$.

By our definitions, we have $P_{\rm o}(F_a)={\deg}_{{\rm in},r_G<0,r_G \text{ is odd.}}(v)$ and $P_{\rm o}(F_b)={\deg}_{{\rm out},r_G<0,r_G \text{ is odd.}}(v)$, have the fact that ${\deg}_{{\rm out},r_G<0,r_G \text{ is odd.}}(v)-{\deg}_{{\rm in},r_G<0,r_G \text{ is odd.}}(v)$ is positive and even, and have the condition (\ref{thm:2.1.2}) of Theorem \ref{thm:2}.\\
\ \\
Case 3 The difference $P_{\rm o}(F_a)-P_{\rm o}(F_b)$ is positive and even..\\
In this case, in the situation of Theorem \ref{thm:2}, we put $(i_a,i_b,r_G(e_{a,j}),r_G(e_{b,j}))$ as before. This guarantees the condition (\ref{thm:2.2}) of Theorem \ref{thm:2}.

We remember the condition (\ref{thm:1.2}) of Theorem \ref{thm:1}. In other words, $P_{\rm o}(F_a) \leq P(F_b)$, where the relation $P_{\rm o}(F_b) \leq P(F_a)$ holds by the condition on this case $P_{\rm o}(F_b)<P_{\rm o}(F_a)$ with $P_{\rm o}(F_a) \leq P(F_a)$. We can discuss as Case 2 and have the condition (\ref{thm:2.1.3}) of Theorem \ref{thm:2}, by the symmetry.\\
\ \\
We have shown that the assumption of Theorem \ref{thm:2} is satisfied. This completes the proof.
\end{proof}
\subsection{A proof of Theorem \ref{thm:1}.}
\begin{proof}[A proof of Theorem \ref{thm:1}]
Theorem \ref{thm:3} implies the sufficiency.

We prove that the conditions (\ref{thm:1.1}) and (\ref{thm:1.2}) also give a necessary condition. This is a main ingredient of our paper. This is not discussed in \cite{kitazawa2}. We assume the existence of a Morse function $\tilde{f}_{F_a,F_b}:{\tilde{M}}_{F_a,F_b} \rightarrow \mathbb{R}$.

The condition (\ref{thm:1.1}) is clear by fundamental properties of closed surfaces.

We prove that the condition (\ref{thm:1.2}) is satisfied.

We prove the relation $P_{\rm o}(F_b) \leq P(F_a)$.

We need fundamental theory of attachments of so-called {\it handles}, corresponding to critical points of Morse functions, naturally. This is presented in \cite{milnor2} for example as classical, fundamental and important theory.

Hereafter, for the boundary of a manifold $X$, let us use $\partial X$. For example, for a $k$-dimensional (standard) disk $D^k$, $\partial D^k=S^{k-1}$.

In terms of handles, we discuss the structure of ${\tilde{M}}_{F_a,F_b}$ as a smooth manifold. For this, \cite[Lemma 6.6]{saeki1} and related arguments on handles are also important. We can understand the manifold ${\tilde{M}}_{F_a,F_b}$ by handles in the following steps.
\begin{itemize}
\item First, we prepare the product $F_a \times D^1=F_a \times [0,1]$ and identify $F_a$ with $F_a \times \{0\}$ by the map $i_{F_a}:F_a \rightarrow F_a \times \{0\}$ defined by $i_{F_a}(x)=(x,0)$.
\item Second, we choose suitable finitely many mutually disjoint copies of $S^1 \times D^1$ and $D^2$ smoothly embedded in $F_a \times \{1\}$. The number of the copies of $D^2$ must be even. Hereafter, we consider these copies of $D^2$ as copies of $D^2 \sqcup D^2$, instead.
\item Third, we attach so-called {\it $2$-handles} $D^2 \times D^1$ to the chosen copies of $S^1 \times D^1$ suitably along $\partial D^2 \times D^1$, one after another. There exists a one-to-one correspondence between 2-handles and singular points of index $2$ of the Morse function. Let the complementary set of $F_a=F_a \times \{0\}$ of the boundary of the resulting $3$-dimensional compact and connected manifold be denoted by $F_s$.% Of course the corner of the resulting $3$-dimensional manifold is smoothed canonically. 

\item Last, we attach so-called {\it $1$-handles} $D^1 \times D^2$ to the chosen copies of $D^2 \sqcup D^2$ suitably along $\partial D^1 \times D^2$, one after another. There exists a one-to-one correspondence between $1$-handles and singular points of index $1$ of the function. The resulting manifold can be regarded to be ${\tilde{M}}_{F_a,F_b}$. %, after the corner is smoothed canonically.
The complementary set of $F_a=F_a \times \{0\}$ of the boundary of the resulting manifold is naturally regarded as $F_b$.
\end{itemize}

We investigate topological relations among $F_a$, $F_b$ and $F_s$. We investigate the values $P(F_a)$, $P(F_b)$ and $P(F_s)$ and the values $P_{\rm o}(F_a)$, $P_{\rm o}(F_b)$ and $P_{\rm o}(F_s)$.

Let the resulting surface obtained by attaching the $j$ handles, from the 1st $2$-handle to the $j$-th $2$-handle here, to $F_a \times \{1\}$, be denoted by $F_{a,j}$. This is, in other words, the complementary set of $F_a=F_a \times \{0\}$ of the boundary of the $3$-dimensional manifold obtained as a result of the attachment of the handles. By applying elementary topological arguments on closed surfaces, we have either of the following.
\begin{itemize}
\item The relation $P(F_{a,j+1})=P(F_{a,j})$ holds and the number of connected components of $F_{a,j+1}$ is greater by $1$ than that of $F_{a,j}$. Here the handle decomposes a connected component of $F_{a,j}$ into two connected summands of it. Furthermore, the relation $P_{\rm o}(F_{a,j+1})=P_{\rm o}(F_{a,j})$ or $P_{\rm o}(F_{a,j+1})=P_{\rm o}(F_{a,j})+2$ holds: the former holds in the case for at least one of the resulting new connected components $F_{a,j+1,{\rm c}}$, $P(F_{a,j+1,{\rm c}})$ is even, and the latter holds in the case both of the two resulting new connected components $F_{a,j+1,{\rm c}_1}$ and $F_{a,j+1,{\rm c}_2}$, the integers $P(F_{a,j+1,{\rm c}_1})$ and $P(a_{s,j+1,{\rm c}_2})$ are odd.
\item The numbers of connected components of $F_{a,j}$ and $F_{a,j+1}$ are same. In addition, either of the following holds.
\begin{itemize}
	\item The relation $P(F_{a,j+1})=P(F_{a,j})$ holds. For exactly one connected component of $F_{a,j}$, the topology is changed: before this change, the connected component of $F_{a,j}$ is orientable and the resulting connected component of $F_{a,j+1}$ is still orientable after the change.
\item The relation $P(F_{a,j+1})=P(F_{a,j})-2$ holds. For exactly one component of $F_{a,j}$, the topology is changed: before this change, the connected component of $F_{a,j}$ is non-orientable and the resulting connected component of $F_{a,j+1}$ is still non-orientable after the change.
\item The relation $P(F_{a,j+1})=P(F_{a,j})-2k$ holds for some positive integer $k$ and a connected component which is orientable appears newly in $F_{a,j+1}$: the connected component is changed from a closed, connected and non-orientable connected component of $F_{a,j}$.  
\end{itemize}
Furthermore, the relation $P_{\rm o}(F_{a,j+1})=P_{\rm o}(F_{a,j})$ holds in this case.
\end{itemize}

We have the following from this argument. \\
\ \\
Claim A  $P(F_{s}) \leq P(F_{a,j}) \leq P(F_a)$.\\
\ \\
Let the resulting surface obtained by attaching the $j$ handles, from the 1st $1$-handle to the $j$-th $1$-handle here, to $F_s$, be denoted by $F_{s,j}$. More precisely, 
$F_a \times \{1\}$ is changed to $F_{s}$ after attaching the $2$-handles, and,
the $1$-handles are attached to the interior of $(F_a \times \{1\}) \bigcap F_s$. The surfaces $F_{s,j}$ and $F_s$ are, in other words, the complementary sets of $F_a=F_a \times \{0\}$ of the $3$-dimensional manifold obtained as a result of the attachment of the handles. As a kind of duality to the previous argument, we have either of the following.
\begin{itemize}
\item The relation $P(F_{s,j+1})=P(F_{s,j})$ holds and the number of connected components of $F_{s,j+1}$ is smaller by $1$ than that of $F_{s,j}$. Here the handle connects two chosen connected components of $F_{s,j}$. Furthermore, the relation $P_{\rm o}(F_{s,j+1})=P_{\rm o}(F_{s,j})$ or $P_{\rm o}(F_{s,j+1})=P_{\rm o}(F_{s,j})-2$ holds: the former holds in the case for at least one of the chosen connected components $F_{s,j,{\rm c}}$, $P(F_{s,j,{\rm c}})$ is even, and the latter holds in the case both of the two chosen connected components $F_{s,j,{\rm c}_1}$ and $F_{s,j,{\rm c}_2}$, the integers $P(F_{s,j,{\rm c}_1})$ and $P(F_{s,j,{\rm c}_2})$ are odd.
\item The numbers of connected components of $F_{s,j}$ and $F_{s,j+1}$ are same. In addition, either of the following holds.
\begin{itemize}
	\item The relation $P(F_{s,j+1})=P(F_{s,j})$ holds. For exactly one connected component of $F_{s,j}$, the topology is changed: before this change, the connected component of $F_{s,j}$ is orientable and the resulting component of $F_{s,j+1}$ is still orientable after the change.
\item The relation $P(F_{s,j+1})=P(F_{s,j})+2$ holds. For exactly one connected component of $F_{s,j}$, the topology is changed: before this change, the connected component of $F_{a,j}$ is non-orientable and the resulting connected component of $F_{a,j+1}$ is non-orientable after the change.
\item The relation $P(F_{s,j+1})=P(F_{s,j})+2k$ holds for some positive integer $k$ and a component which is non-orientable appears newly in $F_{s,j+1}$: the connected component is changed from a closed, connected and orientable connected component of $F_{s,j}$. 
\end{itemize}
Furthermore, the relation $P_{\rm o}(F_{s,j+1})=P_{\rm o}(F_{s,j})$ holds in this case.
\end{itemize}

We have the following from this argument. \\
\ \\
Claim B $P_{\rm o}(F_{b}) \leq P_{\rm o}(F_{s,j}) \leq P_{\rm o}(F_{s})$. \\
\ \\
From our definition, the relation $P_{\rm o}(F_{s}) \leq P(F_{s})$ holds. From Claim A and Claim B, we have the relation $P_{\rm o}(F_{b}) \leq P(F_a)$.

We have $P_{\rm o}(F_{a}) \leq P(F_b)$ similarly, by the symmetry. 

Thus we prove that the condition (\ref{thm:1.2}) is satisfied. 

We have shown the necessity.

This completes our proof.
	\end{proof}

\section{Our conclusion and remarks.}

Theorem \ref{thm:1} implies Problem \ref{prob:1} is solved in the case $m=3$ completely.

We explain some examples.

\begin{Ex}
\begin{enumerate}
	\item Let $F_a$ be a closed, connected and orientable surface and $F_b={\mathbb{R}P}^2 \sqcup {\mathbb{R}P}^2$. 
We have $P(F_a)=P_{\rm o}(F_a)=0$. We also have $P(F_b)=P_{\rm o}(F_b)=2$.
This satisfies the condition (\ref{thm:1.1}) and does not satisfy the condition (\ref{thm:1.2}), in Theorem \ref{thm:1}. 
\item 
Let $F_a$ be a closed, connected and non-orientable surface and $F_b={\mathbb{R}P}^2 \sqcup {\mathbb{R}P}^2$. 
We can see that $P(F_a)$ is positive. Furthermore, $P(F_a)$ is even (odd) if and only if $P_{\rm o}(F_a)=0$ (resp. $P_{\rm o}(F_a)=1$). 

This case satisfies the conditions of Theorem \ref{thm:1} if and only if $P(F_a)>0$ is even.  
\end{enumerate}
\end{Ex}
As another result, the previous sufficient condition of \cite[Main Theorem 2]{kitazawa2} has been simplified to Theorem \ref{thm:3}.

Our result is also regarded as a positive answer to \cite[Remarks 2 and 3]{kitazawa2}, for example. 

Last, we review cases other than our present case for Problem \ref{prob:1}, again.

In \cite{kitazawa3}, the author has given answers to Problem \ref{prob:1} in a specific case for general $m$ (\cite[Main Theorems 1 and 2]{kitazawa3}). This is also on cases where connected components of $F_a$ and $F_b$ are boundaries of compact and connected manifolds of a certain class, the class of {\it MFHs}: the manifolds of the boundaries are of the class of {\it BMFHs}. The author has also completely and affirmatively solved the case
where $F_a$ and $F_b$ are diffeomorphic to $S^{m-1}$ or of the form ${\sharp}_{j=1}^l (S^{k_j} \times S^{m-k_j-1})$ with $m>2$, $l \geq 0$ and $1 \leq k_j \leq m-2$. He has also completely and affirmatively solved the case
 $m=4$ under the constraint that $F_a$ and $F_b$ are orientable there.

Our next problem is, the case $m=4$ where the manifolds $F_a$ and $F_b$ may be non-orientable, for example.

\section{Acknowledgment.}
The author would like to thank Osamu Saeki for fruitful discussions on \cite{saeki2} with \cite{kitazawa1, kitazawa3}. This has motivated the author to present \cite{kitazawa2} and the present work.
The author would also like to thank all related to the conference
(https://www.japanese-singularities.net/workshop/ws20250603.html) for the opportunity to give a talk related to \cite{kitazawa1, kitazawa2, kitazawa3} and the present paper.
\section{Conflict of interest and Data availability.}
\noindent {\bf Conflict of interest.} \\
The author worked at Institute of Mathematics for Industry (https://www.jgmi.kyushu-u.ac.jp/en/about/young-mentors/). This project is closely related to our study. He thanks them for their encouragements. The author is also a researcher at Osaka Central
Advanced Mathematical Institute (OCAMI researcher): this is supported by MEXT Promotion of Distinctive Joint Research Center Program JPMXP0723833165. He is not employed there. However, he also thanks them for such an opportunity.\\
The talk in the conference
(https://www.japanese-singularities.net/workshop/ws20250603.html) is supported by  JSPS KAKENHI Grant Number JP23H01075.\\
%Some of works by other researchers and this version may overlap in some of the contents due to the nature that our problems are natural in theory of Morse functions and applications to differential topology and that related mathematical studies are very fundamental and classical in some senses, for example. However the present version of our paper is presented independent of these work. \\
%Saga Souhatsu Mathematical Seminar (http://inasa.ms.saga-u.ac.jp/Japanese/saga-souhatsu.html), inviting the author as a speaker, is funded and supported by JST Fusion Oriented REsearch for disruptive Science and Technology JPMJFR202U: the author was a speaker on 2024/7/12 supported by this project.\\
\ \\
{\bf Data availability.} \\
We have generated no data other than the present article. Note that the present paper is also seen as a kind of addenda to \cite{kitazawa2}.


\begin{thebibliography}{25}
%	\bibitem{buchstaberpanov} V. M. Buchstaber and T. E. Panov, \textsl{Toric topology}, Mathematical Surveys and Monographs, Vol. 204, American Mathematical Society, Providence, RI, 2015.
%	\bibitem{burletderham} O. Burlet and G. de Rham, \textsl{Sur certaines applications g\'en\'eriques d'une vari\'et\'e close a $3$ dimensions dans le plan}, Enseign. Math. 20 (1974). 275--292.
	%	\bibitem{calabi} E. Calabi, Quasi-surjective mappings and a generation of Morse theory, Proc. U.S.-Japan Seminar in Differential Geometry, Kyoto, 1965, pp. 13--16.
	%
	%		\bibitem{cavicchioli} A. Cavicchioli, \textsl{Covering numbers of manifolds and critical points of a Morse function}, Israel. J. Math. 70 (1990), 279--304.
	% \bibitem{cerf} J. Cerf, \textsl{La stratification naturelle des espaces de fonctions deff\'erentiables r'eelles et le th'eor`eme de la pseudo-isotopie}, Inst. Hautes Etudes Sci. Publ. Math. 39 (1970), 5--173.
	%		\bibitem{choimasudasuh} S. Choi, M. Masuda and D. Y. Suh, \textsl{Topological classification of generalized Bott towers}, Trans. Amer. Math. Soc. 362 (2010), 1097--1112.
	%		\bibitem{cornealuptonopreatanre} O. Cornea, G. Lupton, J. Oprea and D. Tanr\'e, \textsl{Lusternik-Schnirelmann category}, Mathematical Surveys and Monographs, 103, Amer. Math. Soc., Providence, RI, 2003.
	%\bibitem{crowleyescher} D. Crowley and C. Escher, \textsl{A classification of $S^3$-bundles over $S^4$}, Differential. Geom. Appl. 18 (2003), 363--380, arXiv:0004147.
	%\bibitem{crowleynordstrom} D. Crowley and J. Nordstr\"{o}m, \textsl{The classification of $2$-connected $7$-manifolds}, Proc. London. Math. Soc. 119 (2019), 1--54, arXiv:1406.2226.

%	\bibitem{bochnakcosteroy} J. Bochnak, M. Coste and M.-F. Roy, \textsl{Real algebraic geometry}, Ergebnisse der Mathematik und ihrer Grenzgebiete (3) [Results in Mathematics and Related Areas (3)], vol. 36, Springer-Verlag, Berlin, 1998. Translated from the 1987 French original; Revised by the authors.
%		\bibitem{bochnakkucharz} J. Bochnak and W. Kucharz, \textsl{Algebraic approximation of mappings into spheres}, Michigan Mathematical Journal, vol. 34, no. 1, 1987.
%	\bibitem{bodinpopescupampusorea} A. Bodin, P. Popescu-Pampu and M. S. Sorea, \textsl{Poincar\'e-Reeb graphs of real algebraic domains}, Revista Matem\'atica Complutense, https://link.springer.com/article/10.1007/s13163-023-00469-y, 2023, arXiv:2207.06871v2.
%\bibitem{bott} R. Bott, \textsl{Nondegenerate critical manifolds}, Ann. of Math. 60 (1954), 248--261.
%\bibitem{carmesinschulz} S. Carmesin and A. Schulz, \textsl{Arrangements of orthogonal circles with many intersections}, Graph Drawing and Network Visualization (a conference paper), SPRINGER NATURE Link, 2021, arXiv:2106.03557v2. 
%\bibitem{costantino}  F. Costantino, \textsl{A short introduction to shadows of $4$-manifolds}, Fundamenta Mathematicae 251 no. 2 (2005), 427--442.
%\bibitem{costantinothurston} F. Costantino, D. Thurston, \textsl{$3$-manifolds efficiently bound $4$-manifolds}, J. Topol. 1 (2008),
%703--745.
%	\bibitem{delzant} T. Delzant, \textsl{Hamiltoniens p\'eriodiques et images convexes de l'application moment}, Bull. Soc. Math. France 116 (1988), No. 3, 315--339.
%\bibitem{ehresmann} C. Ehresmann, \textsl{Les connexions infinitesimales dans un espace fibre differentiable}, Colloque de Topologie, Bruxelles (1950), 29--55.
%\bibitem{fujitakitabeppumitsuishi} H. Fujita, Y Kitabeppu and A. Mitsuishi, \textsl{Distance functions and convex bodies and symplectic toric manifolds}, arXiv:2003.02293.
%\bibitem{gelbukh} I. Gelbukh, \textsl{Loops in Reeb graphs of $n$-manifolds}, diskrete \& Computational Geometry, 59 (4) (2018), 843--863. 
%%\bibitem{gelbukh2} I. Gelbukh, \textsl{Approximation of Metric Spaces by Reeb Graphs: Cycle Rank of a Reeb Graph, the Co-rank of the Fundamental Group, and Large Components of Level Sets on Riemannian Manifolds}, Filomat (in press), arxiv:1903.00777.
%\bibitem{gelbukh1} I. Gelbukh, \textsl{A finite graph is homeomorphic to the Reeb graph of a Morse-Bott function}, Mathematica Slovaca, 71 (3), 757--772, 2021; doi: 10.1515/ms-2021-0018. 
%\bibitem{gelbukh2} I. Gelbukh, \textsl{Morse-Bott functions with two critical values on a surface}, Czechoslovak Mathematical Journal, 71 (3), 865--880, 2021; doi: 10.21136/CMJ.2021.0125-20. 
%\bibitem{gelbukh3} I. Gelbukh, \textsl{Criterion for a graph to admit a good orientation in terms of leaf blocks}, Monatsh. Math., 198, 61--77, 2022. 
%\bibitem{gelbukh4} I. Gelbukh, \textsl{Realization of a digraph as the Reeb graph of a Morse-Bott function on a given surface}, Topology and its Applications, 2024. 
%\bibitem{gelbukh5} I. Gelbukh, \textsl{Reeb Graphs of Morse-Bott Functions on a Given Surface}, Bulletin of the Iranian Mathematical Society, Volume 50 Article number 84, 2024. 
\bibitem{golubitskyguillemin} M. Golubitsky and V. Guillemin, \textsl{Stable Mappings and Their Singularities}, Graduate Texts in Mathematics (14), Springer-Verlag (1974).
%\bibitem{hempel} J. Hempel, \textsl{3- Manifolds}, AMS Chelsea Publishing, 2004. 
%\bibitem{hiratukasaeki} J. T. Hiratuka and O. Saeki, \textsl{Triangulating Stein factorizations of generic maps and Euler Characteristic formulas}, RIMS Kokyuroku Bessatsu B38 (2013), 61--89. 
%\bibitem{hiratukasaeki2} J. T. Hiratuka and O. Saeki, \textsl{Connected components of regular fibers of differentiable maps}, in "Topics on Real and Complex Singularities", Proceedings of the 4th Japanese-Australian Workshop (JARCS4), Kobe 2011,  World Scientific, 2014, 61--73. 
%\bibitem{ishikawakoda} M. Ishikawa and Y. Koda, \textsl{Stable maps and branched shadows of $3$-manifolds}, Mathematische Annalen 367 (2017), no. 3, 1819--1863, arXiv:1403.0596.
\bibitem{izar} Izar. S. A, \textsl{Funções de Morse e Topologia das Superfícies I: O grafo de Reeb de $f:M \rightarrow \mathbb{R}$}, Métrica no. 31, In Estudo e Pesquisas em Matemática, Brazil: IBILCE, 1988, https://www.ibilce.unesp.br/Home/Departamentos/Matematica/metrica-31.pdf.
%\bibitem{kitazawa1} N. Kitazawa, \textsl{On round fold maps} (in Japanese), RIMS Kokyuroku Bessatsu B38 (2013), 45--59.
%\bibitem{kitazawa1} N. Kitazawa, \textsl{On manifolds admitting fold maps with singular value sets of concentric spheres}, Doctoral Dissertation, Tokyo Institute of Technology (2014).%
%\bibitem{kitazawa2} N. Kitazawa, \textsl{Fold maps with singular value sets of concentric spheres}, Hokkaido Mathematical Journal Vol.43, No.3 (2014), 327--359.
\bibitem{kitazawa1} N. Kitazawa, \textsl{On Reeb graphs induced from smooth functions on $3$-dimensional closed orientable manifolds with finitely many singular values}, Topol. Methods in Nonlinear Anal. Vol. 59 No. 2B, 897--912.

\bibitem{kitazawa2} N. Kitazawa, \textsl{On Reeb graphs induced from smooth functions on $3$-dimensional closed manifolds which may not be orientable}, Methods of Functional Analysis and Topology Vol. 29 No. 1 (2023), 57--72, 2024.

\bibitem{kitazawa3} N. Kitazawa, \textsl{
Constructing Morse functions with given Reeb graphs and level sets}, accepted for publication in Topol. Methods in Nonlinear Anal., arXiv:2108.06913
(, where the title has been changed from the title there), 2025.
%\bibitem{kitazawa4} N. Kitazawa, \textsl{On a classification of Morse functions on $3$-dimensional manifolds represented as connected sums of manifolds of Heegaard genus one}, arXiv:2411.15943, 2024.
%\bibitem{kitazawa2} N. Kitazawa, \textsl{On Reeb graphs induced from smooth functions on closed or open surfaces}, Methods of Functional Analysis and Topology Vol. 28 No. 2 (2022), 127--143, arXiv:1908.04340.
%\bibitem{kitazawa5} N. Kitazawa, \textsl{Real algebraic functions on closed manifolds whose Reeb graphs are given graphs}, Methods of Functional Analysis and Topology Vol. 28 No. 4 (2022), 302--308, arXiv:2302.02339, 2023.
%\bibitem{kitazawa6} N. Kitazawa, \textsl{Explicit construction of explicit real algebraic functions and real algebraic manifolds via Reeb graphs}, Algebraic and geometric methods of analysis 2023 “The book of abstracts”, 49—51, this is the abstract book of the conference "Algebraic and geometric methods of analysis 2023" and published after a short review (https://www.imath.kiev.ua/$\sim$topology/conf/agma2023/), https://imath.kiev.ua/$\sim$topology/conf/agma2023/contents/abstracts/texts/kitazawa/kitazawa.pdf, 2023.
%\bibitem{kitazawa5} N. Kitazawa, \textsl{Notes on explicit special generic maps into Euclidean spaces whose dimensions are greater than $4$}, a revised version is submitted based on positive comments (major revision) by referees and editors after the first submission to a refereed journal, arXiv:2010.10078.

%\bibitem{kitazawa6} N. Kitazawa, \textsl{Round fold maps and the topologies and the differentiable structures of manifolds admitting explicit ones}, submitted to a refereed journal, arXiv:1304.0618.
%\bibitem{kitazawa0.5} N. Kitazawa, \textsl{Constructing fold maps by surgery operations and homological information of their Reeb spaces}, submitted to a refereed journal, arxiv:1508.05630.
%\bibitem{kitazawa0.6} N. Kitazawa, \textsl{Notes on fold maps obtained by surgery operations and algebraic information of their Reeb spaces}, arxiv:1811.04080.


%\bibitem{kitazawa6} N. Kitazawa, \textsl{On Reeb graphs induced from smooth functions on $3$-dimensional closed manifolds which may not be orientable}, a revised version is submitted to a refereed journal after based on positive comments by editors and referees after the second submission to a refreed journal, arXiv:2108.01300.
%\bibitem{kitazawa7} N. Kitazawa, \textsl{Realization problems of graphs as Reeb graphs of Morse functions with prescribed preimages}, submitted to a refereed journal, arXiv:2108.06913.
%\bibitem{kitazawa10} N. Kitazawa,\textsl{Round fold maps on $3$-dimensional manifolds and their integral and rational cohomology rings}, arXiv:2301.07008.
%\bibitem{kitazawa6} N. Kitazawa, \textsl{A class of naturally generalized special generic maps}, arXiv:2212.03174.
%\bibitem{kitazawa7} N. Kitazawa, \textsl{Construction of real algebraic functions with prescribed preimages}, submitted to a refereed journal as the second version based on positive comments by referees and editors, arXiv:2303.00953v3.
%\bibitem{kitazawa7} N. Kitazawa, \textsl{Natural real algebraic maps of non-positive codimensions with prescribed images whose boundaries consist of non-singular real algebraic hypersurfaces satisfying transversality}, a previous version of the present paper, arXiv:2303.10723v8.
%\bibitem{kitazawa6} N. Kitazawa, \textsl{A note on real algebraic maps which are topologically special generic maps}, previous version(s) of the present article and the version arXiv:2303.00953v2 is submitted to a refereed journal, arXiv:2312.10646. 
%\bibitem{kitazawa9} N. Kitazawa, \textsl{Some remarks on real algebraic maps which are topologically special generic maps}, submitted to a refereed journal, arXiv:2312.10646. 
%\bibitem{kitazawa10} N. Kitazawa, \textsl{A note on cohomological structures of special generic maps}, a revised version is submitted based on positive comments by referees and editors after the third submission to a refereed journal.
%\bibitem{kitazawasaeki1} N. Kitazawa and O. Saeki, \textsl{Round fold maps on $3$-manifolds}, accepted for publication after a refereeing process and to appear in Algebraic \& Geometric Topology, arXiv:2105.00974.
		%	\bibitem{kitazawasaeki2} N. Kitazawa and O. Saeki, \textsl{Round fold maps of $n$-dimensional manifolds into ${\mathbb{R}}^{n-1}$}, submitted to a refereed journal, arXiv:2111.13510.
%\bibitem{ishikawakoda} M. Ishikawa, Y. Koda, \textsl{Stable maps and branched shadows of $3$-manifolds}, arXiv:1403.0596.
%\bibitem{kobayashisaeki} M. Kobayashi and O. Saeki, \textsl{Simplifying stable mappings into the plane from a global viewpoint}, Trans. Amer. Math. Soc. 348 (1996), 2607--2636.
%\bibitem{kohnpieneranestadrydellshapirosinnsoreatelen} K. Kohn, R. Piene, K. Ranestad, F. Rydell, B. Shapiro, R. Sinn, M-S. Sorea and S. Telen, \textsl{Adjoints and Canonical Forms of Polypols}, arXiv:2108.11747.
%\bibitem{kollar} J. Koll\'ar, \textsl{Nash's work in algebraic geometry}, Bulletin (New Series) of the American Matematical Society (2) 54, 2017, 307--324.
%\bibitem{kucharz} W. Kucharz, \textsl{Some open questions in real algebraic geometry}, Proyecciones Journal of Mathematics, Vol. 41 No. 2 (2022), Universidad Cat\'olica del Norte Antofagasta, Chile, 437--448.
%\bibitem{martinezalfaromezasarmientooliveira} J. Martinez-Alfaro, I. S. Meza-Sarmiento and R. Oliveira, \textsl{Topological  classification of simple Morse Bott functions on surfaces}, Contemp. Math. 675 (2016), 165--179.%
%\bibitem{marzantowiczmichalak} W. Marzantowicz and L. P. Michalak, \textsl{Relations between Reeb graphs, systems of hypersurfaces and epimorphisms onto free groups}, Fund. Math., 265 (2), 97--140, 2024.

\bibitem{masumotosaeki} Y. Masumoto and O. Saeki, \textsl{A smooth function on a manifold with given Reeb graph}, Kyushu J. Math. 65 (2011), 75--84.
%\bibitem{maciasvirgospereirasaez} E. Mac\'ias-Virg\'os and M. J. Pereira-S\'aez, Height functions on compact symmetric spaces, Monatshefte f\"ur Mathematik 177 (2015), 119--140. 
\bibitem{michalak} L. P. Michalak, \textsl{Realization of a graph as the Reeb graph of a Morse function on a manifold}. Topol. Methods in Nonlinear Anal. 52 (2) (2018), 749--762, arXiv:1805.06727.
%\bibitem{michalak2} L. P. Michalak, \textsl{Combinatorial modifications of Reeb graphs and the realization problem}, Discrete Comput. Geom. 65 (2021), 1038--1060, arXiv:1811.08031.
%\bibitem{michalak3} L. P. Michalak, \textsl{Reeb graph invariants of Morse functions and $3$-manifold groups}, arXiv:2403.02291.
%\bibitem{milnor} J. Milnor, \textsl{Singular points of complex hypersurfacs}, Annals of Mathematics Studies, No. 61, Princeton University Press, Princeton, N. J.; University of Tokyo Press, Tokyo, 1968.
\bibitem{milnor1} J. Milnor, \textsl{Morse Theory}, Annals of Mathematic Studies AM-51, Princeton University Press; 1st Edition (1963.5.1).
\bibitem{milnor2} J. Milnor, \textsl{Lectures on the h-cobordism theorem}, Math. Notes, Princeton Univ. Press, Princeton, N.J. 1965.
%\bibitem{moise} E. E. Moise, \textsl{Affine Structures in $3$-Manifold{\rm :} V. The Triangulation Theorem and Hauptvermutung}, Ann. of Math., Second Series, Vol. 56, No. 1 (1952), 96--114.
%\bibitem{morin} B. Morin, \textsl{Formes canoniques des singulariti\'{e}s d\'{}une application diff\'{e}rentiable}, C. E. Acad. Sci. Paris 260 (1965), 5662--5665, 6503--6506.
%\bibitem{nash} J. Nash, \textsl{Real algbraic manifolds}, Ann. of Math. (2) 56 (1952), 405--421.
%\bibitem{ranicki} A. Ranicki, \textsl{Algebraic and geometric surgery}, https://www.maths.ed.ac.uk/~v1ranick/books/surgery.pdf, 2002.
%\bibitem{ramanujam} S. Ramanujam, \textsl{Morse theory of certain symmetric spaces}, J. Diff. Geom. 3 (1969), 213--229.
\bibitem{reeb} G. Reeb, \textsl{Sur les points singuliers d\'{}une forme de Pfaff compl\'{e}tement int\`{e}grable ou d\'{}une fonction num\'{e}rique}, Comptes Rendus
 Hebdomadaires des S\'{e}ances de I\'{}Acad\'{e}mie des Sciences 222 (1946), 847--849.
%\bibitem{saeki1} O. Saeki, \textsl{Notes on the topology of folds}, J. Math. Soc. Japan Volume 44, Number 3 (1992), 551--566.
%\bibitem{saeki1} O. Saeki, \textsl{Topology of special generic maps of manifolds into Euclidean spaces}, Topology Appl. 49 (1993), 265--293.
%\bibitem{saeki0.2} O. Saeki, \textsl{Topology of singular fibers of differentiable maps}, Lecture Notes in Math., Vol. 1854, Springer-Verlag, 2004. 
\bibitem{saeki1} O. Saeki, \textsl{Morse functions with sphere fibers}, Hiroshima Math. J. Volume 36, Number 1 (2006),  141--170.
\bibitem{saeki2} O. Saeki, \textsl{Reeb spaces of smooth functions on manifolds}, International Mathematics Research Notices, maa301, Volume 2022, Issue 11, June 2022, 3740--3768, https://doi.org/10.1093/imrn/maa301, arXiv:2006.01689.
%\bibitem{saeki3} O. Saeki, \textsl{Reeb spaces of smooth functions on manifolds II}, Res. Math. Sci. 11, article number 24 (2024), https://link.springer.com/article/10.1007/s40687-024-00436-z.
%\bibitem{saekitakase} O. Saeki and M. Takase, \textsl{Desingularizing special generic maps}, Journal of G\"okova Geometry Topology (2013), 1--24.
%\bibitem{sakurai} S. Sakurai, Master Thesis, Kyushu. Univ..
% \bibitem{saekitakase} O. Saeki and M. Takase, \textsl{Desingularizing special generic maps}, Journal of Gokova Geometry Topology 7 (2013), 1--24.
%\bibitem{saeki2} O. Saeki, \textsl{Topology of special generic maps of manifolds into Euclidean spaces}, Topology Appl. 49 (1993), 265--293.
%\bibitem{saeki4} O. Saeki, \textsl{Singular fibers and $4$-dimensional cobordism group}, Pacific J. Math. 248 (2010), 233--256.
%\bibitem{saekisakuma} O. Saeki and K. Sakuma, \textsl{On special generic maps into ${\mathbb{R}}^3$}, Pacific J. Math. 184 (1998), 175--193.
%\bibitem{saekisuzuoka} O. Saeki and K. Suzuoka, \textsl{Generic smooth maps with sphere fibers} J. Math. Soc. Japan Volume 57, Number 3 (2005), 881--902.
\bibitem{sharko} V. Sharko, \textsl{About Kronrod-Reeb graph of a function on a manifold}, Methods of Functional Analysis and
Topology 12 (2006), 389--396.
%\bibitem{shiota} M. Shiota, \textsl{Thom's conjecture on triangulations of maps}, Topology 39 (2000), 383--399.
%\bibitem{sorea1} M. S. Sorea, \textsl{The shapes of level curves of real polynomials near strict local maxima},  Ph. D. Thesis, Universit\'e de Lille, Laboratoire Paul Painlev\'e, 2018.
%\bibitem{sorea2} M. S. Sorea, \textsl{Measuring the local non-convexity of real algebraic curves}, Journal of Symbolic Computation 109 (2022), 482--509.
%\bibitem{sorea1} M. S. Sorea, \textsl{The shapes of level curves of real polynomials near strict local maxima},  Ph. D. Thesis, Universit\'e de %Lille, Laboratoire Paul Painlev\'e, 2018.
%\bibitem{sorea2} M. S. Sorea, \textsl{Measuring the local non-convexity of real algebraic curves}, J. Symbolic Compute. 109 (2022), 482--509.
%\bibitem{stong} R. E. Stong, \textsl{Notes on cobordsm theory}, Princeton Universty Press, 1968.
%\bibitem{takeuchi} M. Takeuchi, \textsl{Nice functions on symmetric spaces}, Osaka. J. Mat. (2) Vol. 6 (1969), 283--289.
%\bibitem{tamaki1} D. Tamaki, Algebraic Topology A Guide to literature,  http://pantodon.jp/index.rb?body=about, 2023. 
%\bibitem{tamaki2} D. Tamaki, Algebraic Topology A Guide to literature (Submanifold arrangement), http://pantodon.jp/index.rb?body=submanifold\_arrangement, 2023.
%\bibitem{tamaki3} D. Tamaki, Algebraic Topology A Guide to literature (Arrangement variations), http://pantodon.jp/index.rb?body=arrangement\_variations, 2023.

%\bibitem{thom} R. Thom, \textsl{Les singularites des applications differentiables}, Ann. Inst. Fourier (Grenoble) 6 (1955-56), 43--87.
%\bibitem{tognoli} A. Tognoli, \textsl{Su una congettura di Nash}, Ann. Scuola Norm. Sup. Pisa (3) 27 (1973), 167--185.
%\bibitem{turaev} Vladimir G. Turaev, \textsl{Topology of shadows}, Preprint, 1991.
%\bibitem{wall} C. T. C Wall, \textsl{Classification problems in differential topology -- {\rm I:} Classificationon handlebodies}, Topology 2 (1963), 253--261.
%\bibitem{wall2} C. T. C. Wall \textsl{Classification problems in differential topology -- {\rm II:} Diffeomorphismsof handlebodies}, Topology 2 (1963), 263--272.
%\bibitem{wall3} C. T. C. Wall, \textsl{Classification problems in differential topology -- {\rm Q:} Quadratic forms on finite groups and related topics}, Topology 2 (1963), 281--298.
%\bibitem{wall4} C. T. C. Wall, \textsl{Classification problems in differential topology -- {\rm III:} Applications to special cases}, Topology 3 (1965), 291--304.
%%\bibitem{wall5} C. T. C. Wall, \textsl{Classification problems in differential topology -- {\rm IV:} Thickenings}, Topology 5 (1966), 73--94.
%\bibitem{wall6} C. T. C. Wall, \textsl{Classification problems in differential topology -- {\rm VI:} Classification of |{\rm (}$s-1${\rm )}-connected {\rm (}$2s+1${\rm )}-manifolds}, Topology 6 (3) (1967), 273--296.
%\bibitem{whitney} H.  Whitney,  \textsl{On singularities of mappings of Euclidean spaces: I,  mappings of the plane into the plane},  Ann.  of Math.  62 (1955),  374--410. 

	%\bibitem{zhubr1} A. V. Zhubr, Closed simply-connected six-dimensional manifolds: proofs of classification theorems, Algebra i Analiz 12 (2000), no. 4, 126--230.
%\bibitem{zhubr2} A. V. Zhubr (responsible for the page), http://www.map.mpim-bonn.mpg.de/6-manifolds:\_1-connected.
\end{thebibliography}
\end{document}